\documentclass[a4paper,USenglish,cleveref,thm-restate]{lipics-v2021}
\hideLIPIcs
\nolinenumbers
\pdfoutput=1
\usepackage[T1]{fontenc}
\usepackage{amsmath,amssymb}
\usepackage{graphicx}
\usepackage{needspace}
\usepackage{subcaption}
\crefname{conjecture}{Conjecture}{Conjectures}
\Crefname{conjecture}{Conjecture}{Conjectures}
\title{How to Construct High Barrycades}

\usepackage{todonotes}
\usepackage{tikz}
\usetikzlibrary{fit,backgrounds,patterns}
\definecolor{palettecolor2}{rgb}{1.0,0.0,0.0}
\definecolor{palettecolor3}{rgb}{0.0,1.0,0.0}
\definecolor{palettecolor4}{rgb}{0.0,0.0,1.0}
\definecolor{palettecolor5}{rgb}{1.0,1.0,0.0}
\definecolor{palettecolor6}{rgb}{0.0,1.0,1.0}
\definecolor{palettecolor7}{rgb}{1.0,0.0,1.0}
\definecolor{palettecolor8}{rgb}{0.6666666666666666,0.0,0.0}
\definecolor{palettecolor9}{rgb}{0.0,0.6666666666666666,0.0}
\definecolor{palettecolor10}{rgb}{0.0,0.0,0.6666666666666666}
\definecolor{palettecolor11}{rgb}{1.0,0.6666666666666666,0.0}
\definecolor{palettecolor12}{rgb}{0.6666666666666666,0.0,1.0}
\definecolor{palettecolor13}{rgb}{1.0,0.6666666666666666,0.0}
\definecolor{palettecolor14}{rgb}{0.0,1.0,0.6666666666666666}
\definecolor{palettecolor15}{rgb}{1.0,0.0,0.6666666666666666}
\definecolor{palettecolor16}{rgb}{0.0,0.6666666666666666,1.0}
\definecolor{palettecolor17}{rgb}{0.6666666666666666,1.0,0.0}
\definecolor{palettecolor18}{rgb}{0.6666666666666666,0.0,0.6666666666666666}
\definecolor{palettecolor19}{rgb}{0.0,0.6666666666666666,0.6666666666666666}
\definecolor{palettecolor20}{rgb}{0.6666666666666666,0.6666666666666666,0.0}

\author{Jakub Binięda}{Institute of Theoretical Computer Science, Faculty of Mathematics and Computer Science, Jagiellonian University, Kraków, Poland}{jakub.binieda@student.uj.edu.pl}{https://orcid.org/0009-0004-5225-5133}{}
\author{Michał Dębski}{Warsaw University of Technology, Warszawa, Poland}{michal.debski@pw.edu.pl}{https://orcid.org/0000-0001-9606-6052}{}
\author{Grzegorz Gutowski}{Institute of Theoretical Computer Science, Faculty of Mathematics and Computer Science, Jagiellonian University, Krak{\'o}w, Poland \and \url{https://grzegorz.gutowscy.pl}}{grzegorz.gutowski@uj.edu.pl}{https://orcid.org/0000-0003-3313-1237}{Partially supported by grant no.~2023/49/B/ST6/01738 from National Science Centre, Poland.}
\author{Mateusz Milewski}{Institute of Theoretical Computer Science, Faculty of Mathematics and Computer Science, Jagiellonian University, Kraków, Poland}{mateusz.milewski@student.uj.edu.pl}{https://orcid.org/0009-0002-6068-9383}{}
\authorrunning{J. Binięda, M. Dębski, G. Gutowski, and M. Milewski}
\Copyright{Jakub Binięda, Michał Dębski, Grzegorz Gutowski, and Mateusz Milewski}

\acknowledgements{
We want to thank Jarosław Grytczuk and Bartłomiej Pawlik for introducing us to the problem.
We want to thank Piotr Białek for suggesting the simulated annealing heuristic and for insightful comments.
}

\newcommand{\floor}[1]{\left\lfloor #1 \right\rfloor}
\newcommand{\ceil}[1]{\left\lceil #1 \right\rceil}

\newcommand{\set}[1]{\left\{ #1 \right\}}
\newcommand{\brac}[1]{\left( #1 \right)}

\let\leq\leqslant

\let\le\leqslant
\let\ge\geqslant

\let\rho\varrho

\definecolor{dark blue}{rgb}{0.121,0.47,0.705}
\let\emph\relax\DeclareTextFontCommand{\emph}{\color{dark blue}\em}

\begin{document}
\maketitle

\begin{abstract}
Given two positive integers \emph{height} $h$ and \emph{order} $n$, 
the barrycade construction problem asks for a set of $h$ permutations of the integers from $1$ to $n$ such that all the proper partial sums given by these permutations are pairwise distinct.
The name barrycade was coined by Richard K.\ Guy and refers to Barry Cipra, who introduced this kind of arrangement problem.
A simple calculation shows that a solution can only exist for $n \ge 2h-2$ and it is conjectured that there always exists a solution for every height~$h \ge 2$ and order~$n \ge 2h-2$.

In this work,
 for every height~$h \ge 1$,
 we present a construction of a barrycade of height~$h$ and order~$n = 2h+3$.
We also present a randomized heuristic that allows us to find a barrycade of height~$h$ and conjectured optimal order~$n=2h-2$, for every height $2 \leq h \leq 50$.
Thus, we confirm the conjectured optimal order for all heights up to $50$.

We also consider a related corral construction problem, where the permutations define a cyclic arrangement.
In this setting,
 for every height~$h \ge 1$,
 we present a construction of a corral of height~$h$ and order~$n=2h$.
A heuristic approach, similar to the one used for barrycades, allows us to find a corral of height~$h$ and conjectured optimal order~$n=2h-1$, for every height $1 \leq h \leq 50$.
We confirm Tomoki Nakamigawa's conjecture on well-dispersed partitions of cyclic groups for the number of parts up to $20$.
\end{abstract}

\section{Introduction}\label{sec:intro}

Let $n$ and $h$ be two positive integers, and let $w=\frac{n(n+1)}{2} = 1 + 2 + \cdots + n$.
Given a set of $h$ permutations $\rho_1,\rho_2,\ldots,\rho_h$ of the set $\set{1,2,\ldots,n}$, for every $1\le i \le h$ and $1 \le j \le n-1$, let $\sigma_{i,j}$ be the $j$-th \emph{proper partial sum} of the $i$-th permutation given by $\sigma_{i,j} = \sum_{k=1}^{j} \rho_i(k)$.
If all the proper partial sums $\sigma_{i,j}$ for every $1 \le i \le h$ and $1 \le j \le n-1$ are pairwise distinct, we say that such a set of permutations forms a \emph{breakfree barrycade} of \emph{order}~$n$, \emph{height}~$h$, and \emph{width}~$w$.
In this paper, we are interested in minimizing the order of a breakfree barrycade of a given height~$h$.
As every proper partial sum satisfies $1 \le \sigma_{i,j} \le w - 1$, for a breakfree barrycade of height~$h$ and order~$n$ we have $h(n-1) \le \frac{n(n+1)}{2}-1$, which for $h\ge 2$ and $n \ge 2$ is equivalent to $n \ge 2h-2$.
It is conjectured that this trivial bound is optimal and that for every $h \ge 2$ and $n \ge 2h-2$ there exists a breakfree barrycade of height~$h$ and order~$n$.

Barry Cipra was the first to ask for such constructions, and the problem was later mentioned by Stan Wagon in his \textit{Problem of the Week}~\cite{Wagon2011} and popularized by Richard K. Guy in his famous presentation \textit{Things I’d Like to Know}~\cite{Guy2016} and in a later paper~\cite{Guy2020}.
Some constructions by Rob Pratt of optimal barrycades of orders up to $26$, and of optimal corrals of orders $26$, $30$, $34$, and $38$, are mentioned by Stan Wagon~\cite{Wagon2011} and Richard K. Guy~\cite{Guy2016,Guy2020}.

It is convenient to visualize a barrycade as a rectangular construction made of rectangular bricks of uniform height $1$ and of various widths.
The construction of height $h$ has $h$ rows of bricks, each row has total width $w=\frac{n(n+1)}{2}$, and each row contains a single brick of each integer width from $1$ to $n$.
In each row, the bricks are placed next to each other from left to right in the order given by the corresponding permutation.
A barrycade is breakfree if no two bricks from different rows have left boundaries of the same $x$-coordinate except for the left boundary of the rectangle.

\begin{figure}[ht]
    \centering
    \begin{tikzpicture}[scale=0.5]
        \begin{small}
        \input{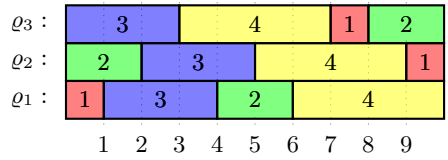}
        \end{small}
    \end{tikzpicture}
    \caption{A breakfree barrycade of height~\(h=3\) and order~\(n=4\) generates every proper partial sum from $1$ to $9$.}
    \label{fig:barrycade}
\end{figure}
For a small example, let $n = 4$ and observe that the permutation $(1, 3, 2, 4)$ gives proper partial sums $1, 4, 6$, while $(2, 3, 4, 1)$ gives proper partial sums $2, 5, 9$, and $(3, 4, 1, 2)$ gives proper partial sums $3, 7, 8$; see \cref{fig:barrycade} for a visual representation of the construction.
As every proper partial sum given by these permutations is distinct, the construction gives a breakfree barrycade of height~$h=3$ and order~$n=4$.
As every permutation of $3$ elements gives $2$ proper partial sums and there are only $5$ distinct possible proper partial sums, there are no breakfree barrycades of height~$h=3$ and order~$n=3$.

We also consider a cyclic variant of the problem in which the right end of the last brick in each row is identified with the left end of the first brick.
For a formal definition, fix positive integers $h$, $n$, and $w=\frac{n(n+1)}{2}$.
Let $\rho_1,\rho_2,\ldots,\rho_h$ be $h$ permutations of the set $\set{1,2,\ldots,n}$ and $\tau_1,\tau_2,\ldots,\tau_h$ be integers called \emph{shifts}.
For every $1\le i \le h$ and $1 \le j \le n$, we define the \emph{cyclic partial sum} $\sigma_{i,j}$ as the residue modulo $w$ of the $j$-th partial sum of the $i$-th permutation plus $\tau_i$, i.e., $\sigma_{i,j} = \brac{\tau_i + \sum_{k=1}^{j} \rho_i(k)} \bmod w$.
The permutations with shifts form a \emph{breakfree corral} of \emph{order} $n$, \emph{height} $h$, and \emph{width} $w$ if all the cyclic partial sums $\sigma_{i,j}$ for every $1 \le i \le h$ and $1 \le j \le n$ are pairwise distinct.
As each shift and permutation gives $n$ cyclic partial sums, we have $hn \le w = \frac{n(n+1)}{2}$, which is equivalent to $n \ge 2h-1$.
Similarly to the barrycade problem, it is conjectured that for every integer $h \ge 1$, there exists a breakfree corral of height~$h$ and order~$n=2h-1$.

\begin{figure}[ht]
    \centering
    \begin{tikzpicture}[scale=0.5]
        \begin{small}
        \input{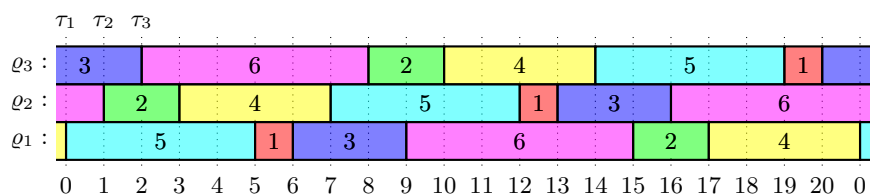}
        \end{small}
    \end{tikzpicture}
    \caption{A breakfree corral of height~\(h=3\) and order~\(n=6\) generates every possible cyclic partial sum from $0$ to $20$ except for $4$, $11$, and $18$.}
    \label{fig:corral}
\end{figure}
See \cref{fig:corral} for an example of such a construction, where the permutation $(5,1,3,6,2,4)$ with shift $0$ gives cyclic partial sums $5, 6, 9, 15, 17, 0$, while $(2,4,5,1,3,6)$ with shift $1$ gives cyclic partial sums $3, 7, 12, 13, 16, 1$, and $(6,\allowbreak 2,\allowbreak 4,\allowbreak 5,\allowbreak 1,\allowbreak 3)$ with shift $2$ gives cyclic partial sums $8, 10, 14, 19, 20, 2$.
As all partial sums are distinct, the construction gives a breakfree corral of height~$h=3$ and order~$n=6$.

This work is dedicated to the following two conjectures mentioned earlier.
\begin{conjecture}\label{cnj:corral}
    For all positive integers $h$ and $n\ge 2h-1$, there exists a breakfree corral of height~$h$ and order~$n$.
\end{conjecture}
\begin{conjecture}\label{cnj:barrycade}
    For all positive integers $h\ge2$ and $n\ge 2h-2$, there exists a breakfree barrycade of height~$h$ and order~$n$.
\end{conjecture}

For a barrycade of height~$h$ and of optimal order $n=2h-2$, equality holds for ${w-1=h(n-1)}$.
Thus, a breakfree barrycade of optimal order uses every integer coordinate from $1$ to $w-1$ as a proper partial sum.
Alex Fink, see~\cite{Guy2016,Guy2020}, suggested imposing an additional balancing condition.
Partition the interval of proper partial sums into $n-1$ consecutive sections
\begin{equation*}
    I_j=\set{(j-1)h+1,(j-1)h+2,\ldots,jh},
    \qquad j=1,2,\ldots,n-1.
\end{equation*}
A permutation is \emph{balanced} if exactly one of its proper partial sums belongs to each section $I_j$; a barrycade is \emph{balanced} if each of its permutations is balanced.
Geometrically, if we partition the brick wall into segments of width $h$, then each such segment contains exactly one end of a brick in each row.
This stronger, balanced form of the problem was proposed as a possible route to constructing optimal barrycades.

Tomoki Nakamigawa~\cite{Nakamigawa2015} studied the cyclic version of this balancing condition in the language of well-dispersed partitions of cyclic groups.
For a corral of height~$h$ and of optimal order $n=2h-1$, the width is $w=h(2h-1)$, and the partial cyclic sums in $h$ rows form a partition of the cyclic group $\mathbb{Z}_w$ into $h$ sets whose common cyclic gap multiset is $\set{1,2,\ldots,2h-1}$.
The partition is \emph{well-dispersed} if each of these sets contains exactly one coordinate in every one of the $2h-1$ consecutive blocks of size $h$, which is precisely the cyclic counterpart of the balancing condition.
We say that a corral is \emph{balanced} if the corresponding partition of $\mathbb{Z}_w$ is well-dispersed. 
Tomoki Nakamigawa~\cite[Theorem~5]{Nakamigawa2015} constructed explicit balanced corrals for $h=2,5,6,7$ and showed a procedure that given an optimal balanced corral of height $h$ constructs an optimal balanced corral of height $2h$.
This gives a construction of a balanced breakfree corral of height~$h$ and of optimal order~$n=2h-1$ for every $h=a\cdot 2^b$, where $a\in\set{2,5,6,7}$ and $b\ge0$.
Consequently, his work gives breakfree corrals of optimal order for infinitely many heights, and, together with \cref{thm:corral}, it implies that \cref{cnj:corral} holds for all these heights.

Michał Dębski, Jarosław Grytczuk, Paweł Naroski, Bartłomiej Pawlik, Jakub Przybyło, and Małgorzata Śleszyńska-Nowak~\cite{DebskiGNPPS2026} studied both finite and infinite variants of the barrycade problem.
They proved that there exists a constant $c>0$ such that, for infinitely many positive integers $n$, there is a breakfree barrycade of order~$n$ and height at least $cn$.
They also proposed constructions for the infinite variant.
Their techniques inspired the constructions presented in this paper.

The paper is organized as follows.
In \cref{sec:corral} we present a construction of breakfree corrals that shows the following.
\begin{restatable}{theorem}{thmcorral}\label{thm:corral}
    For all positive integers $h$ and $n\ge 2h$, there exists a breakfree corral of height~$h$ and order~$n$.
\end{restatable}
In \cref{sec:barrycade} we present a construction of breakfree barrycades that shows the following.
\begin{restatable}{theorem}{thmbarrycade}\label{thm:barrycade}
    For all positive integers $h$ and $n\ge 2h+3$, there exists a breakfree barrycade of height~$h$ and order~$n$.
\end{restatable}
In \cref{sec:sa} we present a randomized heuristic that allows us to find breakfree barrycades and corrals of conjectured optimal orders for all heights up to $50$, which, together with \cref{thm:corral}, confirms \cref{cnj:corral} for all heights up to $50$.
Using the same techniques, we extend Tomoki Nakamigawa's result and give explicit constructions of balanced corrals for $h=9,11,13,15,17,18,19$.
This supports a stronger version of \cref{cnj:corral} given by Tomoki Nakamigawa~\cite[Problem~2]{Nakamigawa2015}, that says that for every $h \neq 3$ there is a balanced breakfree corral of height~$h$ and order~$n=2h-1$.
The constructed barrycades and corrals are available in a public repository~\cite{repo} together with a verification program that checks if a given barrycade or corral is breakfree.

\subsection{Notation}\label{sec:notation}

\subsubsection*{Fragmentary solutions}
In order to present our constructions, we need to describe fragmentary solutions that we call \emph{fences}.
A fence fixes the shifts and some prefixes of the permutations.
Assuming that values of height~$h$, order~$n$, and width $w=\frac{n(n+1)}{2}$ are fixed, we can describe a fragmentary solution by a sequence of $h$ descriptions, one for each row.
We use the following notation
\begin{equation*}
    i:{}^t[x_1][x_2]\ldots[x_j]{}^u
\end{equation*}
with $x_1,x_2,\ldots,x_j$ being distinct integers from the set $\set{1,2,\ldots,n}$,
to denote that the fence fixes shift $\tau_i=t$ and the first $j$ values of the partial permutation $\rho_i$ to be $\rho_i(1)=x_1$, $\rho_i(2)=x_2$, $\ldots$, $\rho_i(j)=x_j$.
Additionally, we have $t + x_1+x_2+\ldots+x_j=u$, and we use the notation $\upsilon_i = u$ to denote this partial sum generated by the fence.

In the brick representation, this means that in the $i$-th row, the bricks of widths $x_1$, $x_2$, up to $x_j$ are located next to each other from left to right, starting at $x$-coordinate $t$ and ending at $x$-coordinate $u$.
A fence with $j$ bricks $x_1, x_2, \ldots, x_j$ in the $i$-th row generates $j+1$ partial sums: $\tau_i$, $\tau_i+x_1$, $\tau_i+x_1+x_2$, $\ldots$, $\tau_i+x_1+x_2+\ldots+x_j = \upsilon_i$.
We call the sequence $(\tau_1,\tau_2,\ldots,\tau_h)$ to be the \emph{left boundary} of the fence, and the sequence $(\upsilon_1,\upsilon_2,\ldots,\upsilon_h)$ to be the \emph{right boundary}.
We say that a fence is \emph{breakfree} if all the partial sums defined by the fence are pairwise distinct.
We will always make sure that the constructed fences are breakfree, so that we can extend it to a complete breakfree corral or barrycade.
We say that a fence is \emph{left-skewed} if the shifts satisfy
$\tau_{i+1} = \tau_i + 1$ for every $i=1,2,\ldots,h-1$.
It is \emph{right-skewed} if we have $\upsilon_{i+1} = \upsilon_i + 1$ for every $i=1,2,\ldots,h-1$.
It is \emph{skewed} if it is both left-skewed and right-skewed.

\subsubsection*{Mono fence}
\begin{figure}[ht]
    \centering
    \begin{tikzpicture}[scale=0.33]
    \begin{tiny}
    \input{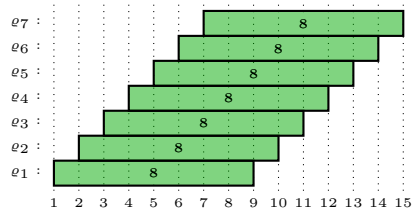}
    \end{tiny}
    \end{tikzpicture}
    \caption{The $8$-mono fence of height $7$.}
    \label{fig:mono_fence}
\end{figure}
For a fixed value of height $h$, and any integer $k \ge h$, let the \emph{$k$-mono fence} be a fence given by the following description (see \cref{fig:mono_fence} for a visual representation).
\begin{gather*}
    i:{}^i[k]{}^{k+i}\quad\text{; for } i=1,2,\ldots,h\text{.}
\end{gather*}
It is easy to see that the left boundary of the $k$-mono fence is $1,2,\ldots,h$ and the right boundary is $k+1,k+2,\ldots,k+h$.
There are no other partial sums generated by the $k$-mono fence, and as $k+1>h$, the $k$-mono fence is breakfree and skewed.

\subsubsection*{Operations}

There are some natural operations that can be applied to fences to obtain new fences.
Given a fence with shifts $\tau_1,\tau_2,\ldots,\tau_h$ and partial permutations $\rho_1,\rho_2,\ldots,\rho_h$, we define the following operations.
\subparagraph*{Shift}
Given an integer $s$, we can \emph{shift} a fence by adding $s$ to each shift $\tau_i$.
It is clear that if the original fence is breakfree, then the shifted fence is also breakfree.
\subparagraph*{Rotation}
We can \emph{rotate} a fence by moving shifts and partial permutations down by one row cyclically, i.e., the first row becomes the last row, and every other row moves down by one.
It is clear that if the original fence is breakfree, then the rotated fence is also breakfree.
\subparagraph*{Concatenation}
We can \emph{concatenate} two fences of the same height by placing the second one to the right of the first one, so that the last brick of each row in the first fence is adjacent to the first brick of the corresponding row in the second fence.
When doing so, we have to make sure that:
\begin{itemize}
    \item the right boundary of the first fence matches the left boundary of the second fence,
    \item in each row, the two fences use bricks of distinct widths,
    \item all the partial sums in the resulting fence are pairwise distinct.
\end{itemize}
\subparagraph*{Safe concatenation}
We observe that if $F_1$ is a breakfree, right-skewed fence, and $F_2$ is a breakfree, left-skewed fence, then there exists a unique integer $s$ such that right boundary of $F_1$ matches the left boundary of $F_2$ shifted by $s$.
Let $F$ be a fence resulting from concatenation of $F_1$ and $F_2$ shifted by $s$.
As right boundary of $F_1$ is a sequence of $h$ consecutive integers, and $F_1$ is breakfree, we get that every other partial sum generated by $F_1$ is a smaller number.
As left boundary of $F_2$ shifted by $s$ is the same sequence of $h$ consecutive integers we get that every other partial sum generated by $F_2$ shifted by $s$ is a larger number.
We conclude that $F$ is breakfree.
Thus, in order to safely concatenate a breakfree, right-skewed fence with a breakfree, left-skewed fence and obtain a new breakfree fence it is enough to make sure that the two fences use bricks of distinct widths in each row.

\section{Corrals}\label{sec:corral}
In this section we present a construction of breakfree corrals that proves \cref{thm:corral}.
As the construction for $h=1$ is trivial, we fix some positive integer $h \ge 2$ and we will construct a breakfree corral of order~$n=2h$ and height~$h$.

\subsubsection*{Base fence}
Consider the fence of height~$h$ and order~$n=2h$ described by the following description (see \cref{fig:corral_base_fence} for a visual representation).
\begin{gather*}
    1:{}^1[2h]^{2h+1}\text{,}\\
    i:{}^i[2h-2i+2][2i-2]^{2h+i}\quad\text{; for } i=2,\ldots,\ceil{\frac{h}{2}}\text{,}\\
    i:{}^i[2h-2i+1][2i-1]^{2h+i}\quad\text{; for } i=\ceil{\frac{h}{2}}+1,\ldots,h\text{.}
\end{gather*}
\begin{figure}[ht]
    \begin{subfigure}{0.5\textwidth}
    \centering
    \begin{tikzpicture}[scale=0.225]
    \begin{tiny}
    \input{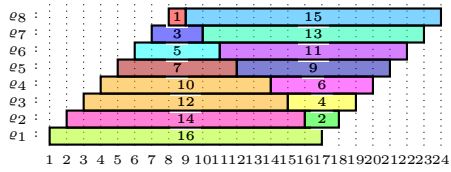}
    \end{tiny}
    \end{tikzpicture}
    \caption{The base fence for $h=8$.}
    \label{fig:corral_base_fence_even}
    \end{subfigure}%
    \begin{subfigure}{0.5\textwidth}
    \centering
    \begin{tikzpicture}[scale=0.25]
    \begin{tiny}
    \input{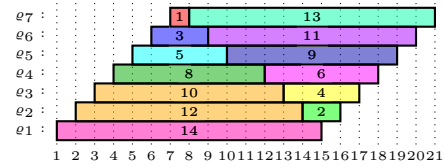}
    \end{tiny}
    \end{tikzpicture}
    \caption{The base fence for $h=7$.}
    \label{fig:corral_base_fence_odd}
    \end{subfigure}
    \caption{Base fences for even and odd $h$.}
    \label{fig:corral_base_fence}
\end{figure}

We call this fence the \emph{base fence}.
Observe the following properties of the base fence:
\begin{itemize}
    \item Each row in the base fence contains bricks of total width \(2h\).
    One brick of width \(2h\) is used in the first row, and two bricks of total width $2h$ are used in every other row.
    \item The base fence uses $2h-1$ bricks of distinct widths, i.e., one brick of every width from $1$ to $2h$ except for the brick of width $h$ that is missing.
    \item The base fence is skewed.
    \item The base fence is breakfree. To see that, observe that the partial sums generated by the base fence are: $1,\ldots,h$ on the left boundary, $2h+1,\ldots,3h$ on the right boundary, and $h-1$ distinct values between $h+1$ and $2h$ except for the value $h+1+\floor{\frac{h}{2}}$.
\end{itemize}

\subsubsection*{Rotating fence}
\begin{figure}[ht]
    \centering
    \begin{tikzpicture}[scale=0.33]
    \begin{tiny}
    \input{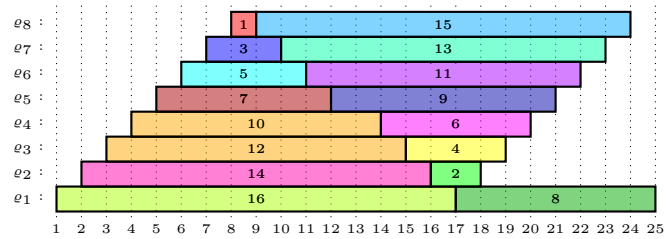}
    \end{tiny}
    \end{tikzpicture}
    \caption{The rotating fence for $h=8$.}
    \label{fig:corral_rotating_fence}
\end{figure}
Consider the fence given by the base fence with brick of width $h$ added to the end of the first row (just after the brick of width $2h$).
Such a fence is still breakfree, and it uses all the bricks of widths $1$ to $2h$ -- each once.
We call this fence the \emph{rotating fence} (see \cref{fig:corral_rotating_fence} for a visual representation).
The rotating fence is left-skewed, but not right-skewed.
Further, if we rotate the rotating fence, then the resulting fence is right-skewed, but no longer left-skewed.

\subsubsection*{Final construction}
We are now ready to construct a breakfree corral of height~$h$ and order~$n=2h$ and prove \cref{thm:corral}.

\thmcorral*
\begin{proof}
\begin{figure}[ht]
    \centering
    \begin{tikzpicture}[scale=0.315]
    \begin{tiny}
    \input{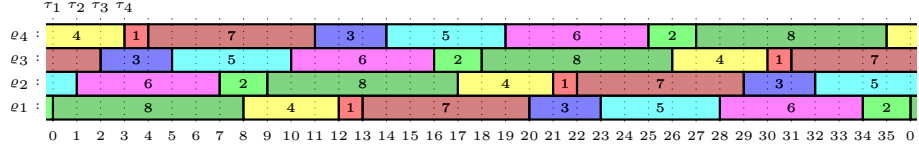}
    \end{tiny}
    \end{tikzpicture}
    \caption{A breakfree corral of order~\(n=8\) and height~\(h=4\).}
    \label{fig:corral_final}
\end{figure}
    The theorem trivially holds for $h=1$.
    For a fixed $h \ge 2$, we first prove the theorem for $n=2h$.
    We start the construction with the rotating fence.
    We rotate the rotating fence and we observe that the result is right-skewed.
    Then we apply the following procedure $h-1$ times, keeping the result right-skewed after each iteration.
    For a single repetition, we perform a safe concatenation of the current right-skewed construction with a new left-skewed copy of the rotating fence, and a single rotation of the result.
    It is easy to see that the result is again right-skewed.

    For every $i=0,1,\ldots,h-1$, let $F_i$ denote the construction after $i$ repetitions of the procedure.
    We show that, for $i=0,1,\ldots,h-1$, every row in $F_i$ is a concatenation of $i+1$ cyclically consecutive rows in the rotating fence.
    This is clearly true for $i=0$, as $F_0$ is the rotated rotating fence.
    As in each step $i$, we make a single rotation after a safe concatenation with a new copy of the rotating fence, we get that each row that previously was a concatenation of $i$ cyclically consecutive rows is extended by the exact next row and is now a concatenation of $i+1$ cyclically consecutive rows.
    As the rotating fence uses each brick of width from $1$ to $2h$ exactly once,
    we get that each brick in the $j$-th row of $F_i$ is distinct.
    Therefore, we really can perform the safe concatenation operation.

    After $h-1$ repetitions of the procedure, we have a total of $h$ copies of the rotating fence in the construction.
    We have performed $h$ rotations, so the resulting construction is again left-skewed and right-skewed.
    As each row in the result contains a different row in the rotating fence in each copy of the rotating fence, each row contains a single brick of every width from $1$ to $2h$.
    As we only perform rotations and safe concatenations, it is clear that the resulting fence is breakfree.
    Thus, we have constructed a breakfree, skewed fence of height~$h$ and order~$n=2h$ that uses each brick of width from $1$ to $2h$ exactly once in each row.
    The left boundary of the resulting fence is $1,2,\ldots,h$ and the right boundary is $w+1,w+2,\ldots,w+h$.
    Therefore, we can wrap the resulting fence into a corral by using the exact permutations and shifts defined by the fence.
    When doing so, we merge the left boundary with the right boundary and do not create any conflicts among other bricks.
    Indeed, the left boundary already occupies every integer in $\set{1,2,\ldots,h}$ and the right boundary occupies every integer in $\set{w+1,w+2,\ldots,w+h}$, so, as the fence is breakfree, no other partial sum lies in either of these two intervals.
    As all the partial sums are integers between $1$ and $w+h$, the reduction modulo~$w$ identifies exactly $i$ with $w+i$ for $i=1,2,\ldots,h$.
    We conclude that we have constructed a breakfree corral of height~$h$ and order~$n=2h$ (see \cref{fig:corral_final} for a visual representation of the construction for $h=4$).

    Now, for order $n>2h$, let the \emph{padding fence} be a fence constructed by safe concatenations of $(2h+1)$-mono fence, $(2h+2)$-mono fence, \ldots, and $n$-mono fence.
    The padding fence is a breakfree skewed fence.
    We can construct a breakfree corral of height~$h$ and order~$n$ by wrapping a safe concatenation of the padding fence with the fence constructed for order $n=2h$.
\end{proof}

An implementation of the construction used in the proof of \cref{thm:corral}, together with instructions for running the construction and verification programs, is available in the repository~\cite{repo}.

\section{Barrycades}\label{sec:barrycade}
The construction of breakfree barrycades is similar to the construction of breakfree corrals.
We fix some positive integer $h\ge2$ and construct a breakfree barrycade of height~$h$ and order~$n=2h+3$.
First, we construct a fence of height~$h$ and order~$n=2h+2$ in a manner similar to the construction used in the proof of \cref{thm:corral}.
Then, we show how to modify the constructed fence in order to obtain a breakfree barrycade.

\subsubsection*{Base fence}
First consider the following fence of height~$h+1$:
\begin{equation*}
    i:{}^i[2h-2i+4][2i-1]^{i+2h+3} \quad\text{; for } i=1,\ldots,h+1
\end{equation*}
and remove the row that contains the bricks of widths $h$ and $h+3$.
Finally, we reindex the remaining rows from $1$ to $h$, and set the shift $\tau_i=i$ for every row $i=1,2,\ldots,h$.
This gives us a fence of height~$h$ that we call the \emph{base fence}.
\begin{figure}[ht]
    \centering
    \begin{tikzpicture}[scale=0.33]
    \begin{tiny}
    \input{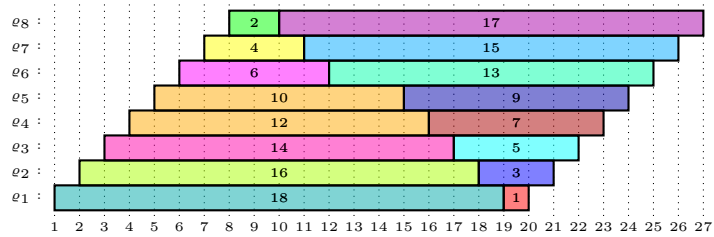}
    \end{tiny}
    \end{tikzpicture}
    \caption{The base fence for $h=8$.}
    \label{fig:barrycade_base_fence}
\end{figure}

Observe the following properties of the base fence:
\begin{itemize}
    \item Each row in the base fence contains bricks of total width \(2h+3\).
    \item The base fence uses $2h$ bricks of distinct widths, i.e., one brick of every width from $1$ to $2h+2$ except for the bricks of widths $h$ and $h+3$ that are missing.
    \item The base fence is skewed.
    \item The base fence is breakfree. All the partial sums generated by the base fence are: $1,\ldots,h$ on the left boundary, $2h+4,\ldots,3h+3$ on the right boundary, and $h$ distinct values between $h+2$ and $2h+3$.
\end{itemize}

\subsubsection*{Rotating fence}
As in the previous construction, we can add the missing brick of width $h$ to the first row in the base fence and obtain a rotating fence.
As in the previous construction, the rotating fence is left-skewed, but not right-skewed and the rotated rotating fence is right-skewed, but not left-skewed.

\subsubsection*{Final construction}
As in the previous construction, we can start the construction with a rotated copy of the rotating fence and perform $h-1$ iterations of a safe concatenation with a copy of the rotating fence and a rotation.
This way we obtain a breakfree skewed fence of height~$h$ and order~$n=2h+2$.
In the resulting construction, each row misses exactly one brick of width $h+3$.
We can add the missing bricks by a safe concatenation of the $(h+3)$-mono fence to obtain a breakfree, skewed fence of height~$h$ and order~$n=2h+2$ that uses each brick of width from $1$ to $2h+2$ exactly once in each row.

We are now ready to construct a breakfree barrycade of order~$n=2h+3$ and height~$h$, proving \cref{thm:barrycade}.

\thmbarrycade*
\begin{proof}
\begin{figure}[ht]
    \centering
    \begin{tikzpicture}[scale=0.255]
    \begin{tiny}
    \input{img/construction_barrycade_9}
    \end{tiny}
    \end{tikzpicture}
    \caption{A breakfree barrycade of order~\(n=9\) and height~\(h=3\).}
    \label{fig:barrycade_final}
\end{figure}
    The theorem trivially holds for $h=1$.
    For a fixed $h\ge2$, we start with the constructed breakfree fence of height~$h$ and order~$2h+2$.
    We observe that in each row $i=2,3,\ldots,h$, the brick of width $i-1$ is in one of the copies of the rotating fence next to the brick of width $2h-i+4$, as $(i-1)+(2h-i+4)=2h+3$ and $i-1\le h-1$, so these two bricks form a row of the base fence other than the removed one.
    We replace these two bricks with a single brick of width $2h+3$.
    Clearly, we have only removed one partial sum from each row except the first one, and the result is still a breakfree fence.
    This way we obtain a breakfree skewed fence of height~$h$ and order~$2h+3$, that misses the brick of width $2h+3$ in the first row, and two bricks of widths $i-1$ and $2h-i+4$ in every row $i=2,3,\ldots,h$.

    If $n>2h+3$, we construct the \emph{padding fence} as a safe concatenation of $(2h+4)$-mono fence, $(2h+5)$-mono fence, \ldots, and $n$-mono fence.
    The padding fence is a breakfree skewed fence.
    We perform a safe concatenation of the padding fence with the constructed breakfree fence of height~$h$ and order~$2h+3$.

    We have now constructed a breakfree fence of height~$h$ and order~$n$ that misses the brick of width $2h+3$ in the first row, and two bricks of widths $i-1$ and $2h-i+4$ in every row $i=2,3,\ldots,h$.
    We finish the construction by adding the missing bricks to each row.
    We add the missing brick of width $2h+3$ to the first row at the end.
    In every other row $i=2,3,\ldots,h$, we add the missing brick of width $i-1$ at the beginning of the row, and the brick of width $2h-i+4$ at the end of the row.
    This way, we have altered the left boundary and the right boundary of the construction from skewed to aligned.
    The final result is a breakfree barrycade of order~$n$ and height~$h$.
\end{proof}

An implementation of the construction used in the proof of \cref{thm:barrycade}, together with instructions for running the construction and verification programs, is available in the repository~\cite{repo}.

\section{Simulated Annealing}\label{sec:sa}
To search for breakfree barrycades and corrals of conjectured optimal orders, we used standard simulated annealing methods~\cite{KirkpatrickGV1983,AartsKM2005}.
For fixed height~$h$ and order~$n$, a state consists of $h$ permutations of $\set{1,2,\ldots,n}$.
For corrals, the shifts were fixed at $\tau_i=i-1$ for $i=1,2,\ldots,h$ throughout the search.
The objective is to eliminate collisions between partial sums of different rows: proper partial sums for barrycades and cyclic partial sums for corrals.
A natural objective function counts pairs of equal partial sums, so its value is zero exactly when the state represents a breakfree construction.
For heights~$h\le20$, we additionally required the constructions to be balanced and modified the scoring function to penalize violations of the balancing conditions as well as collisions between partial sums.
The only exception is the corral of height~$h=3$, as Tomoki Nakamigawa~\cite{Nakamigawa2015} showed that there is no balanced corral of height~$3$ and order~$5$.

The basic state moves exchange two bricks within a single row or reverse the order of a contiguous segment of a row's permutation.
Both moves preserve the requirement that each row is a permutation and allow the search to modify brick boundaries while preserving the height, order, and width of the construction.
In the standard annealing scheme, a proposed move that does not increase the objective is accepted, whereas a move that increases it by~$\Delta>0$ is accepted with probability~$\exp(-\Delta/T)$, where~$T>0$ is the current temperature.
The temperature is gradually decreased, allowing the search to escape local minima early in a run and favouring improvements as the run progresses.

We ran several variants and implementations of the heuristic, adjusting the parameters across runs until we obtained all the constructions reported in this paper.
The resulting constructions are available in the repository~\cite{repo}, including barrycades and corrals of conjectured optimal orders for heights up to~$50$.
Each saved construction can be checked independently using the provided verification program.
Thus, the heuristic is used only to find candidates; the breakfree property of each reported construction is verified directly from its permutations and, for corrals, its shifts.
Instructions for running the verification programs can be found in the repository~\cite{repo}.
The verification checks cover the implemented theorem constructions for heights up to~$50$ and verify the height, conjectured optimal order, and breakfree property of the saved heuristic constructions in the same range.
They also check balance for heights up to~$20$, except for the corral of height~$3$, which cannot be balanced.

\section{Final Remarks}\label{sec:final}
We have shown that for every positive integer $h$, there exists a breakfree barrycade of height~$h$ and order~$2h+3$ and a breakfree corral of height~$h$ and order~$2h$.
These bounds are close to the respective conjectured optimal orders of $2h-2$ and $2h-1$.
We have also confirmed the existence of barrycades and corrals of conjectured optimal orders for every height $h \le 50$.
Additionally, there are balanced barrycades of optimal order for every height $2 \le h \le 20$, and balanced corrals of optimal order for every height $h \le 20$ other than $h=3$, for which no balanced corral exists~\cite{Nakamigawa2015}.
These small constructions support stronger variants of the conjectures.

The corral conjecture is known to hold for infinitely many heights by Tomoki Nakamigawa's construction~\cite{Nakamigawa2015}.
For barrycades, it remains unknown whether the conjectured optimal order can be attained for infinitely many heights.
Resolving both conjectures remains a very intriguing question.

\bibliography{barrycades}

\end{document}